\documentclass{amsart}
\usepackage{geometry}
\usepackage{enumitem,xcolor,amssymb}

\usepackage{setspace,graphicx}
\usepackage{cite}
\theoremstyle{theorem}
\newtheorem{theorem}{Theorem}[section]
\newtheorem{corollary}[theorem]{Corollary}
\newtheorem{lemma}[theorem]{Lemma}
\theoremstyle{proposition}
\newtheorem{proposition}[theorem]{Proposition}
\theoremstyle{definition}

\theoremstyle{remark}

\newcommand{\R}{\mathbb{R}}
\newcommand{\Sc}{\mathcal{S}}
\newcommand{\Z}{\mathbb{Z}}

\newcommand{\inpr}[3][]{\left\langle#2 \,,\, #3\right\rangle_{#1}}

\newcommand{\oldinpr}[3][]{{\lll\!#2 \,,\, #3\!\ggg_{#1}}}

\newcommand{\oldnorm}[2][]{{|\!\|#2\|\!|_{#1}}}

\title{Products and Convolutions in Hermite-Sobolev spaces}

\author{Suprio Bhar}
\address{Suprio Bhar, Department of Statistics and Data Science, Indian Institute of Technology Kanpur, Kalyanpur, Kanpur - 208016, India.}
\email{suprio@iitk.ac.in}

\author{Rajeev Bhaskaran}
\address{Rajeev Bhaskaran, Harish-Chandra Research Institute, Chhatnag Rd, Jhusi, Prayagraj, Uttar Pradesh - 211019, India.}
\email{brajeev58j@gmail.com}
\begin{document}

\begin{abstract}
In this paper, we show that the product, or equivalently the convolutions of two functions in the Hermite-Sobolev spaces $\mathcal{S}_p(\mathbb{R}^d)$  is again in the same space, for $p$ depending on the dimension $d$. As a consequence, for such $p$ we show that the product, or equivalently the convolutions of $\phi \in \mathcal{S}_p(\mathbb{R}^d)$ and $\psi \in \mathcal{S}_{-p}(\mathbb{R}^d)$ is in $\mathcal{S}_{-p}(\mathbb{R}^d)$. As a further consequence, we show that the operators of translation by $x$ from $\mathcal{S}_p(\mathbb{R}^d)$ to $\mathcal{S}_q(\mathbb{R}^d)$ are bounded uniformly in $x \in \mathbb{R}^d$ for $q \leq -p$ and integers $p > \frac{d}{2}$.
\end{abstract}


\maketitle

\section{Introduction}
Stochastic PDEs with solutions in the dual of Nuclear spaces, more specifically, in the spaces of distributions are well-known in the literature \cite{Ito-monograph, Kallianpur-Xiong}. In this regard, the Hermite-Sobolev spaces $\mathcal{S}_p(\mathbb{R}^d), p \in \R$ (see \cite{Ito-monograph}) can be used to analyze Stochastic PDEs in the space of tempered distributions. In \cite{Rajeev-IJPAM}, certain non-linear operators were used, as coefficients in the Stochastic PDE. For example, the drift coefficient had the form (see \cite[p. 232]{Rajeev-IJPAM})
\[L\phi = \frac{1}{2} \langle \sigma, \phi \rangle^2 \frac{d^2}{dx^2}\phi - \langle b, \phi \rangle \frac{d}{dx}\phi.\]
Here, the non-linearities arise due to scalar multiplication by $\langle \sigma, \phi \rangle^2$ and $\langle b, \phi \rangle$.

On the other hand, if we wish to consider PDEs or SPDEs, like Burger's equation or linear equations of the type considered in \cite{Arvind-Monotonicity} in the framework of Hermite-Sobolev spaces $\mathcal{S}_p(\mathbb{R}^d), p \in \R$, then we need to consider multiplication of the form $(\phi_1, \phi_2) \mapsto \phi_1\phi_2$ in these spaces. For example, in the case of Burgers' equation
\[\frac{\partial}{\partial t} u + u  \frac{\partial}{\partial x} u = 0, \, u(0, x) = f(x)\]
the non-linear term is in the form $\phi\frac{d}{dx} \phi$. This motivates the contents of this paper and we show that the Hermite-Sobolev spaces $\mathcal{S}_p(\mathbb{R}^d)$ are algebras for $p$ depending on the dimension $d$ (see Theorem \ref{main-result-2}). A similar result in the framework of $W^{2,s}$, the classical Sobolev spaces, is known in the literature (see \cite{BrezisBk, Kato-article, PalaisBk}). It is possible that the results on the convolution algebra structure for the Hermite-Sobolev spaces may also be derived from the literature of Shubin Classes and Wiener Amalgams (see \cite{brandenburg75,feichtinger81,luef-rahbani11} and the references therein). Our results possibly offer a more direct approach.

As an application of Theorem \ref{main-result-2}, we also extend the operation of multiplication of tempered distributions by smooth functions to multiplication by functions from appropriate Hermite-Sobolev spaces (see Corollary \ref{main-result-3}). This allows us to consider products and convolutions of a function in $\Sc_p$ with a distribution in $\Sc_{-p}$.  As a further consequence, in Theorem \ref{translation-bound} we establish that the operators $\tau_x$ of translation by $x \in \R^d$ are uniformly bounded from $\Sc_p(\R^d)$ to $\Sc_q(\R^d)$ for integers $p > \frac{d}{2}$ and real number $q \leq -p$.

By Theorem \ref{main-result-2}, the product term in the Burgers' equation above makes sense in $\Sc_{p - 1}$ for $\phi \in \Sc_p$, for $p$ sufficiently large. One may therefore look for solutions of such equations in $\Sc_p$ spaces. Clearly, this is not specific to the Burgers' equation.

We briefly discuss our approach, which seems to be new. First, using the well-known invariance of the Hermite-Sobolev spaces $\Sc_p$ under the Fourier transform, we reduce the problem of showing that the product $\phi_1\phi_2$ is in $\Sc_p$ to one of showing a convolution $\psi_1\ast\psi_2$ is in $\Sc_p$ (see Theorem \ref{fourier-convolution}). This, by definition of the specific choice of norms on $\Sc_p$, reduces to showing that functions of the form
\[x^\alpha \partial^\beta(\psi_1\ast\psi_2)\]
are in $\mathcal{L}^2$ for multi-indices $\alpha$ and $\beta$ satisfying $|\alpha| + |\beta| \leq 2p$. 

At this stage, rather than splitting the derivative term $\partial^\beta(\psi_1\ast\psi_2)$ over the two factors $\psi_1$ and $\psi_2$, like $\partial^{\beta - k}\psi_1$ and $\partial^k \psi_2$, and using the regularities of $\psi_1$ and $\psi_2$, we put all the derivatives on one of the factors, say $\psi_2$, to get $x^\alpha (\psi_1\ast\partial^\beta\psi_2)$.

To obtain $\mathcal{L}^2$-norm bounds for $x^\alpha (\psi_1\ast\partial^\beta\psi_2)$, it suffices (by a Binomial expansion) to estimate $\mathcal{L}^2$-norms of 
\[y \mapsto \int_{\R^d} (y - x)^{\alpha-\gamma} \,\psi_1(y - x) \, x^\gamma \partial^\beta \psi_2(x)\, dx\] for all multi-indices $\alpha, \beta, \gamma$ satisfying $0 \leq \gamma \leq \alpha, |\alpha| + |\beta| \leq 2p$ (see Lemma \ref{equivalent}). Now, at this stage we use Jensen's inequality to get the estimate on the $\mathcal{L}^2$-norms of the above functions. Note that, we can use either of  $|\cdot|^{\alpha-\gamma} |\partial^\beta \psi_2(\cdot)|$ or $|\cdot|^\gamma |\psi_1|$ as the $\mathcal{L}^1$ factor in Jensen's inequality (see Proposition \ref{Jensen}).

The resulting upper bounds are finite, provided the multi-indices $\alpha, \beta$ and $\gamma$ satisfy $|\alpha - \gamma| + |\beta| + r \leq 2p$ or $|\gamma| + r \leq 2p$ (see Theorem \ref{main-result-1}) with $r$ given by Lemma \ref{L1-bound-by-Sp}. That one of the two inequalities holds is verified in the proof of Theorem \ref{main-result-2}.

\section{Notations and Main Results}

\subsection{Preliminaries: Topology on Schwartz space}

Let $\Sc(\mathbb{R}^d)$ denote the space of complex valued rapidly decreasing smooth functions on $\R^d$, with the dual space $\Sc^\prime(\mathbb{R}^d)$ the space of tempered distributions. Let $\mathbb{Z}^d_+:=\{\alpha=(\alpha_1,\cdots, \alpha_d): \; \alpha_i \text{ non-negative integers}\}$. If $\alpha\in\mathbb{Z}^d_+$, we define $|\alpha|:=\alpha_1+\cdots+\alpha_d$. The topology on $\Sc(\mathbb{R}^d)$ is given by a family of seminorms $p_n, n = 0, 1, 2, \cdots$ (see \cite{Ito-monograph, FollandBk}) where
\[p_n(f) := \sup_{x \in \R^d} \left[(1 + \|x\|)^n \max_{\alpha : |\alpha| \leq n} \left|\frac{\partial^{|\alpha|}}{\partial_{x_1}^{\alpha_1}\cdots \partial_{x_d}^{\alpha_d}}f(x) \right|\right], \forall f \in \Sc(\mathbb{R}^d)\]
where $\|x\|$ denotes the usual Euclidean norm for $x \in \R^d$.

For $p \in \R$, consider the increasing family of norms $\oldnorm[p]{\cdot}$, defined by the inner products
\begin{equation}
\oldinpr[p]{f}{g}
:=\sum_{n\in\mathbb{Z}^d_+}(2|n|+d)^{2p}\langle f,h_n\rangle \overline{\langle g,h_n\rangle},\ \ \ f,g\in\Sc(\mathbb{R}^d).
\end{equation}
In the above equation, $\{h_n: n\in\mathbb{Z}^d_+\}$ is an orthonormal basis for $\mathcal{L}^2(\R^d,dx)$ given by the Hermite functions and $\langle\cdot,\cdot\rangle$ is the usual
inner product in $\mathcal{L}^2(\R^d,dx)$. The Hermite-Sobolev spaces $\Sc_p(\mathbb{R}^d), p \in \R$ are defined as the completion of $\Sc(\mathbb{R}^d)$ in
$\oldnorm[p]{\cdot}$. Note that the dual space $\Sc_p^\prime(\mathbb{R}^d)$ is isometrically isomorphic with $\Sc_{-p}(\mathbb{R}^d)$ for $p\geq 0$. We also have $\Sc(\mathbb{R}^d) = \bigcap_{p}(\Sc_p(\mathbb{R}^d),\oldnorm[p]{\cdot}), \Sc^\prime(\mathbb{R}^d) = \bigcup_{p>0}(\Sc_{-p}(\mathbb{R}^d),\oldnorm[-p]{\cdot})$ and $\Sc_0(\mathbb{R}^d) = \mathcal{L}^2(\R^d)$. The following basic relations hold for the $\Sc_p(\mathbb{R}^d)$ spaces: for $0<q<p$, \[\Sc(\mathbb{R}^d)\subset\Sc_p(\mathbb{R}^d)\subset\Sc_q(\mathbb{R}^d)\subset\mathcal L^2(\R^d)=\Sc_0(\mathbb{R}^d)\subset\Sc_{-q}\subset\Sc_{-p}(\mathbb{R}^d)\subset\Sc^\prime(\mathbb{R}^d).\]
The topology on $\Sc(\R^d)$ given by the norms $\oldnorm[p]{\cdot}, p \in \Z_+$ is the same as the usual topology on $\Sc(\R^d)$ (see \cite[Proposition 1.1]{Rajeev-Seminaire}).

For $x \in \R^d$, recall that the translation operator $\tau_x : \Sc(\R^d) \to \Sc(\R^d)$ is given by
\[(\tau_x \phi)(y) := \phi (y - x), \forall y \in \R^d\]
for any $\phi \in \Sc(\R^d)$ and is extended via duality to $\tau_x : \Sc^\prime(\R^d) \to \Sc^\prime(\R^d)$ as a linear operator
\[\inpr{\tau_x \phi}{\psi} = \inpr{\phi}{\tau_{-x} \psi}, \, \forall \phi \in \Sc^\prime(\R^d), \psi \in \Sc(\R^d)\]
We also recall that $\tau_x : \Sc_p(\R^d) \to \Sc_p(\R^d)$ is a bounded linear operator for all $x \in \R^d$ and $p \in \R$ (see \cite[Theorem 2.1]{BR-ST-Heat}).

\subsection{A family of alternative inner-products generating the usual topology on the Schwartz space}

For any non-negative integer $p$, define for $\phi, \psi \in \Sc(\R^d)$,
\[\inpr[p]{\phi}{\psi}:= \sum_{|\alpha| + |\beta|\leq 2p}\int_{\R^d} y^\alpha\partial^\beta \phi(y)\, \overline{y^\alpha\partial^\beta \psi(y)}\, dy,\]
for $\alpha,\beta\in\Z^d_+$. Here, $y^\alpha = y_1^{\alpha_1} \cdots y_d^{\alpha_d}$ and $|y|^\alpha = |y_1|^{\alpha_1} \cdots |y_d|^{\alpha_d}$. Completing $\Sc(\R^d)$ with the above inner-product gives us the Hermite-Sobolev spaces $\Sc_p, p \in \Z_+$. We denote the corresponding norms by $\|\cdot\|_p$. It is known that the topology on $\Sc$ generated by $\|\cdot\|_p, p = 0, 1, \cdots$ is the same as the usual topology. We recall the next result.

\begin{proposition}[{\cite[Proposition 3.3]{BR-ST-Heat}}]\label{norm-equivalence}
For all nonnegative integers $p$, there exist constants $C_1 = C_1(p) > 0$ and $C_2 = C_2(p) > 0$ such that
\[\oldnorm[p]{\phi} \leq C_1 \|\phi\|_p \leq C_2 \oldnorm[p]{\phi}, \forall \phi \in \Sc(\mathbb{R}^d).\]
\end{proposition}

\subsection{Main results}
We require the following well-known integrability condition in our arguments. By a transformation into polar co-ordinates, we have
\begin{equation}\label{integrability}
\int_{\R^d} \frac{1}{(1 + \|x\|^2)^r}\, dx = C_d \int_0^\infty \frac{t^{d-1}}{(1 + t^2)^r} \, dt < \infty,    
\end{equation}
for some positive constant $C_d$, depending on $d$,
provided $r > \frac{d}{2}$.  In the rest of this article, unless stated otherwise, we consider $p$ is a non-negative integer.

\begin{lemma}\label{L1-bound-by-Sp}
For any multi-index $\alpha, \beta$, we have
\[\||x|^\alpha \partial^\beta \psi\|_{\mathcal{L}^1(\R^d)} \leq C_{d, r} \|\psi\|_p, \forall \psi \in \Sc_p(\R^d),\]
provided $r > \frac{d}{2}$ and $|\alpha| + |\beta| + r \leq 2p$. Here, $C_{d, r}$ is a positive constant depending only on $d$ and $r$.
\end{lemma}

\begin{proof}
Applying Cauchy-Schwarz inequality, we have
\begin{align*}
  \int_{\R^d} |x|^\alpha |\partial^\beta \psi(x)|\, dx &= \int_{\R^d} \frac{1}{(1 + \|x\|^2)^{\frac{r}{2}}} (1 + \|x\|^2)^{\frac{r}{2}} |x|^\alpha |\partial^\beta \psi(x)|\, dx\\
  &\leq \left(\int_{\R^d} \frac{1}{(1 + \|x\|^2)^r} \, dx\right)^\frac{1}{2}  \left(\int_{\R^d}(1 + \|x\|^2)^r |x|^{2\alpha} |\partial^\beta \psi(x)|^2\, dx\right)^{\frac{1}{2}}\\
  &\leq \left(\int_{\R^d} \frac{1}{(1 + \|x\|^2)^r} \, dx\right)^\frac{1}{2}  \left(\int_{\R^d}(1 + \|x\|^2)^{r + |\alpha|} |\partial^\beta \psi(x)|^2\, dx\right)^{\frac{1}{2}}
\end{align*}
The statement follows using  \eqref{integrability}, provided $r > \frac{d}{2}$ and $|\alpha| + |\beta| + r \leq 2p$.
\end{proof}

Recall the Fourier transform on $\Sc(\R^d)$ (see \cite{ThangaveluBk}). For $\phi \in \Sc(\R^d)$, define $\hat\phi: \R^d \to \mathbb{C}$
\[\hat \phi (x) := \left(\frac{1}{2\pi}\right)^{\frac{d}{2}}\ \int_{\R^d} e^{-i x.y} \phi(y)\, dy\]
where $i$ denotes the complex number $\sqrt{-1}$. The Fourier transform preserves $\Sc(\R^d)$ and is extended by duality to $\Sc^\prime(\R^d)$. The inverse Fourier transform $\check \phi$ is defined in a similar manner.

Unless stated otherwise, for $\phi_1, \phi_2 \in \Sc^\prime(\R^d)$, we set $\psi_1 = \check \phi_1$ and $\psi_2 = \check \phi_2$.

\begin{theorem}\label{fourier-convolution}
\begin{enumerate}[label=(\roman*)]
    \item (\cite{ThangaveluBk}) The space of Hermite-Sobolev spaces $\Sc_p(\R^d)$ are invariant under Fourier transform. Moreover, 
    \[C_{1, p, d} \|\phi\|_p \leq \|\hat \phi\|_p \leq C_{2, p, d} \|\phi\|_p, \forall \phi \in \Sc_p(\R^d),\]
for some positive constants $C_{1, p, d}$ and $C_{2, p, d}$ depending only on $p$ and $d$.

    \item For any $\phi_1, \phi_2 \in \Sc_p(\R^d)$, we have $\phi_1 \phi_2 = \hat \psi_1 \hat \psi_2 = \widehat{\psi_1\ast\psi_2}$.

    \item Given $\phi_1, \phi_2 \in \Sc_p(\R^d)$, $\phi_1 \phi_2 \in \Sc_p(\R^d)$ if and only if $\psi_1\ast\psi_2 \in \Sc_p(\R^d)$.
\end{enumerate}

\end{theorem}

\begin{corollary}
Given $\phi_1, \phi_2 \in \Sc_p(\R^d)$, $\phi_1 \phi_2 \in \Sc_p(\R^d)$ if and only if $y \mapsto y^\alpha \partial^\beta (\psi_1\ast\psi_2)(y) \in \mathcal{L}^2(\R^d)$ for all multi-indices $\alpha, \beta$ satisfying $|\alpha| + |\beta| \leq 2p$.
\end{corollary}

Note that 
\begin{equation}\label{expansion}
y^\alpha \partial^\beta (\psi_1\ast\psi_2)(y) = \sum_{0 \leq \gamma \leq \alpha} \binom{\alpha}{\gamma} \int_{\R^d} (y - x)^{\gamma} \,\psi_1(y - x) \, x^{\alpha-\gamma} \partial^\beta \psi_2(x)\, dx,
\end{equation}
where $x^\gamma = x_1^{\gamma_1} \cdots x_d^{\gamma_d}$ and $\binom{\alpha}{\gamma} = \prod_{j = 1}^d \binom{\alpha_j}{\gamma_j}$. This observation leads to the next result.

\begin{lemma}\label{equivalent}
Given $\phi_1, \phi_2 \in \Sc_p(\R^d)$, $\phi_1 \phi_2 \in \Sc_p(\R^d)$ if and only if \[y \mapsto \int_{\R^d} (y - x)^{\gamma} \,\psi_1(y - x) \, x^{\alpha-\gamma} \partial^\beta \psi_2(x)\, dx \in \mathcal{L}^2(\R^d)\] for all multi-indices $\alpha, \beta, \gamma$ satisfying $0 \leq \gamma \leq \alpha, |\alpha| + |\beta| \leq 2p$.
\end{lemma}

Applying Jensen's inequality, we have the following upper bound for the $\mathcal{L}^2(\R^d)$ norm of the function $y \mapsto \int_{\R^d} (y - x)^{\gamma} \,\psi_1(y - x) \, x^{\alpha-\gamma} \partial^\beta \psi_2(x)\, dx$.

\begin{proposition}\label{Jensen}
For all multi-indices $\alpha, \beta, \gamma$ satisfying $0 \leq \gamma \leq \alpha, |\alpha| + |\beta| \leq 2p$, we have the following $\mathcal{L}^1$-$\mathcal{L}^2$ estimate:
\begin{align*}
  &\left\| \int_{\R^d} (\cdot - x)^{\gamma} \,\psi_1(\cdot - x) \, x^{\alpha-\gamma} \partial^\beta \psi_2(x)\, dx\right\|_{\mathcal{L}^2(\R^d)}\\
  &\leq \min\left\{\left\||\cdot|^{\alpha-\gamma} \partial^\beta \psi_2(\cdot)\right\|_{\mathcal{L}^1(\R^d)}\times  \left\||\cdot|^\gamma \psi_1(\cdot)\right\|_{\mathcal{L}^2(\R^d)},\right.\\ 
  &\qquad\left.\left\||\cdot|^{\alpha-\gamma} \partial^\beta \psi_2(\cdot)\right\|_{\mathcal{L}^2(\R^d)}\times  \left\||\cdot|^\gamma \psi_1(\cdot)\right\|_{\mathcal{L}^1(\R^d)}\right\}.
\end{align*}

\end{proposition}

\begin{proof}
Jensen's inequality implies the following inequality
\[ \left( \int_{\R^d} f(x)\times g(x)\, dx \right)^2 \leq \|g\|_{\mathcal{L}^1(\R^d)} \int_{\R^d} (f(x))^2 \, |g(x)|\, dx.\]
Now,
\begin{align*}
  &\left\| \int_{\R^d} (\cdot - x)^{\gamma} \,\psi_1(\cdot - x) \, x^{\alpha-\gamma} \partial^\beta \psi_2(x)\, dx\right\|_{\mathcal{L}^2(\R^d)}^2\\
  &= \int_{\R^d} \left[\int_{\R^d} (y - x)^{\gamma} \,\psi_1(y - x) \, x^{\alpha-\gamma} \partial^\beta \psi_2(x)\, dx\right]^2 dy  \\
  &\leq \left\||\cdot|^{\alpha-\gamma} \partial^\beta \psi_2(\cdot)\right\|_{\mathcal{L}^1(\R^d)} \int_{\R^d} \int_{\R^d} \left[(y - x)^{\gamma} \,\psi_1(y - x)\right]^2 x^{\alpha-\gamma} \partial^\beta \psi_2(x)\, dx dy\\
  &= \left\||\cdot|^{\alpha-\gamma} \partial^\beta \psi_2(\cdot)\right\|_{\mathcal{L}^1(\R^d)} \int_{\R^d} \left[\int_{\R^d} \left[(y - x)^{\gamma} \,\psi_1(y - x)\right]^2\, dy \right] x^{\alpha-\gamma} \partial^\beta \psi_2(x)\, dx\\
  &= \left\||\cdot|^{\alpha-\gamma} \partial^\beta \psi_2(\cdot)\right\|_{\mathcal{L}^1(\R^d)}^2\  \left\||\cdot|^\gamma \psi_1(\cdot)\right\|_{\mathcal{L}^2(\R^d)}^2.
\end{align*}
This establishes one of the upper bounds. The other one follows by exchanging the roles of $\mathcal{L}^1(\R^d)$ and $\mathcal{L}^2(\R^d)$ functions in the first estimate.
\end{proof}

Combining Lemma \ref{L1-bound-by-Sp}, Lemma \ref{equivalent}
and Proposition \ref{Jensen}, we have the next result.

\begin{theorem}\label{main-result-1}
Fix $\phi_1, \phi_2 \in \Sc_p(\R^d)$. 
\begin{enumerate}[label=(\alph*)]
    \item For $\phi_1 \phi_2 \in \Sc_p(\R^d)$ to hold, it is sufficient that for all multi-indices $\alpha, \beta, \gamma$ satisfying $0 \leq \gamma \leq \alpha, |\alpha| + |\beta| \leq 2p$, one of the following conditions is true:
\begin{enumerate}[label=(\roman*)]
    \item $\left\||\cdot|^{\alpha-\gamma} \partial^\beta \psi_2(\cdot)\right\|_{\mathcal{L}^1(\R^d)} \leq C \|\phi_2\|_p$ for some positive constant $C = C_{p, d}$

    \item $\left\||\cdot|^\gamma \psi_1(\cdot)\right\|_{\mathcal{L}^1(\R^d)} \leq C \|\phi_1\|_p$ for some positive constant $C = C_{p, d}$
\end{enumerate}

\item The sufficient condition in (a) is implied by the following condition: for all multi-indices $\alpha, \beta, \gamma$ satisfying $0 \leq \gamma \leq \alpha, |\alpha| + |\beta| \leq 2p$, there exists $r = r(\alpha, \beta, \gamma) > \frac{d}{2}$ such that one of the following conditions is true,
\begin{enumerate}[label=(\roman*)]
    \item $|\alpha - \gamma| + |\beta| + r \leq 2p$.

    \item $|\gamma| + r \leq 2p$.

\end{enumerate}
\end{enumerate}
\end{theorem}

\begin{proof}
By Lemma \ref{equivalent}, $\phi_1 \phi_2 \in \Sc_p(\R^d)$ if and only if \[y \mapsto \int_{\R^d} (y - x)^\alpha \,\psi_1(y - x) \, x^\gamma \partial^\beta \psi_2(x)\, dx \in \mathcal{L}^2(\R^d)\] for all multi-indices $\alpha, \beta, \gamma$ satisfying $0 \leq \gamma \leq \alpha, |\alpha| + |\beta| \leq 2p$. To ensure this for all such combinations of $(\alpha, \beta, \gamma)$, we show that one of the upper bounds of $\mathcal{L}^2(\R^d)$ norm appearing on the right hand side of Proposition \ref{Jensen} is further upper bounded by $\|\phi_1\|_p \|\phi_2\|_p$.

By the definition of $\|\cdot\|_p$ and Theorem \ref{fourier-convolution}, we have 
\[\left\||\cdot|^\gamma \psi_1(\cdot)\right\|_{\mathcal{L}^2(\R^d)} \leq  \|\phi_1\|_p, \qquad \left\||\cdot|^{\alpha-\gamma} \partial^\beta \psi_2(\cdot)\right\|_{\mathcal{L}^2(\R^d)} \leq C \|\phi_2\|_p\]
Then $\phi_1 \phi_2 \in \Sc_p(\R^d)$ holds, provided the sufficient condition in part (a) of the statement holds.

To establish the result in part (b), note that by Lemma \ref{L1-bound-by-Sp},
\[\left\||\cdot|^{\alpha-\gamma} \partial^\beta \psi_2(\cdot)\right\|_{\mathcal{L}^1(\R^d)} \leq C \|\phi_2\|_p,\]
for some positive constant $C = C_{p, d, r}$ provided there exists $r = r(\alpha, \beta, \gamma) > \frac{d}{2}$ such that $|\alpha - \gamma| + |\beta| + r \leq 2p$. Similarly,
\[ \left\||\cdot|^\gamma \psi_1(\cdot)\right\|_{\mathcal{L}^1(\R^d)} \leq C \|\phi_1\|_p,\]
for some positive constant $C = C_{p, d, r}$ provided there exists $r = r(\alpha, \beta, \gamma) > \frac{d}{2}$ such that $|\gamma| + r \leq 2p$. The result follows.
\end{proof}

\begin{theorem}\label{main-result-2}
Let $p$ be any integer with $p > \frac{d}{2}$. Then for all $\phi_1, \phi_2 \in \Sc_p(\R^d)$,
\begin{equation}\label{product-continuity}
\|\phi_1 \phi_2\|_p \leq C \|\phi_1\|_p \|\phi_2\|_p,    
\end{equation}
for some positive constant $C = C_{p, d}$, i.e. $\Sc_p(\R^d)$ is an algebra.
\end{theorem}

\begin{proof}
Consider all combinations of $(\alpha, \beta, \gamma)$ satisfying $0 \leq \gamma \leq \alpha, |\alpha| + |\beta| \leq 2p$. Note that $0 \leq |\gamma| \leq |\alpha|$ and $|\alpha - \gamma| = |\alpha| - |\gamma|$.

Set $r := \frac{1}{2}(p + \frac{d}{2})$. Then, $\frac{d}{2} < r < p $. We show that for this choice of $r$, one of $(i)$ or $(ii)$ in Theorem \ref{main-result-1} holds. Consequently, by \eqref{expansion}, \eqref{product-continuity} follows.

Whenever $0 \leq |\gamma| \leq 2p - r$, we have $|\gamma| + r \leq 2p$, satisfying $(ii)$. Therefore $(ii)$ holds for all combinations of $(\alpha, \beta, \gamma)$ with $0 \leq |\gamma| \leq 2p - r$.

Now, consider the case $|\gamma| > 2p - r$. Then 
\[|\alpha- \gamma| + |\beta| + r = (|\alpha| - |\gamma|) + |\beta| + r = (|\alpha| + |\beta|) - |\gamma| + r < 2p - (2p - r) + r = 2r < 2p,\]
satisfying $(i)$.

This completes the proof.
\end{proof}

The following corollary allows us to consider multiplication of tempered distributions in $\Sc_{-p}(\R^d)$ by functions in $\Sc_p(\R^d)$, for an appropriate $p$.
\begin{corollary}\label{main-result-3}
Let $p$ be any integer with $p > \frac{d}{2}$. Let $\eta \in \Sc_{-p}(\R^d)$ and $\psi \in \Sc_{p}(\R^d)$. Then, $\psi \eta , \psi\ast\eta \in \Sc_{-p}(\R^d)$ with
\[\| \psi \eta \|_{-p} \leq C_{p, d} \|\eta\|_{-p}\ \|\psi\|_p, \quad \|\psi \ast \eta \|_{-p} \leq C_{p, d}^\prime \|\eta\|_{-p}\ \|\psi\|_p,\]
for positive constants $C_{p, d}$ and $C_{p, d}^\prime$.
\end{corollary}
\begin{proof}
 For any $\phi \in \Sc(\R^d)$, we have
\[|\left\langle \psi\eta, \phi \right\rangle| = |\left\langle \eta, \psi\phi \right\rangle|\leq \|\eta\|_{-p}\ \|\psi\phi\|_{p} \leq C_{p, d} \|\eta\|_{-p}\ \|\psi\|_p\ \|\phi\|_p,\]
where $C_{p, d}$ is as in Theorem \ref{main-result-2}. The first inequality follows by the identity
\[\| \psi \eta \|_{-p} = \sup_{\overset{\phi \in \Sc(\R^d),}{\|\phi\|_p = 1}} |\left\langle \psi\eta, \phi \right\rangle|\]
The proof of the second inequality in the statement follows from 
Theorem \ref{fourier-convolution}.
\end{proof}

Using the above corollary, we have the following uniform bound for the translation operators $\tau_x, x \in \R^d$ improving the bound of \cite[Theorem 2.1]{BR-ST-Heat}.
\begin{theorem}\label{translation-bound}
We have
\[\sup_{x \in \R^d} \|\tau_x\|_{\Sc_p \to \Sc_q} < \infty\]
whenever $p$ is any integer greater than $\frac{d}{2}$ and $q$ is any real number with $q \leq -p$.
\end{theorem}
\begin{proof}
For any $\phi \in \Sc_p(\R^d)$, note that $\tau_x \phi = \phi \ast \delta_x$. Using Corollary \ref{main-result-3} we have,
\[\sup_{x \in \R^d} \|\tau_x \phi\|_{-p} \leq C_{p, d}^\prime \|\phi\|_p \sup_{x \in \R^d} \|\delta_x\|_{-p}\]
By \cite[Theorem 4.1 a)]{BR-ST-forward}, $\sup_{x \in \R^d} \|\delta_x\|_{-p} < \infty$ and hence 
\[\sup_{x \in \R^d} \|\tau_x\|_{\Sc_p \to \Sc_{-p}} < \infty\]
The stated result for the case $q \leq -p$ follows from the inclusion $\Sc_{-p} \hookrightarrow \Sc_q$.

\end{proof}

\textbf{Acknowledgement:} The authors would like to thank S. Thangavelu for his comments on an earlier draft. We also thank H. Feichtinger for drawing our attention to the related literature on Shubin Classes and Wiener Amalgams.

\bibliographystyle{plain}

\bibliography{ref}
\end{document}